\documentclass[12pt,twoside]{article}
\usepackage{graphicx,graphics,amsthm,amsfonts,amsbsy,amssymb,amsmath,cite,array,arydshln,hyperref,float,tikz}
\usepackage[utf8]{inputenc}
\usepackage[T1]{fontenc}
\hypersetup{hidelinks}
\usepackage[margin=1in]{geometry}
\usepackage{booktabs,longtable}
\usetikzlibrary{positioning,fit}
\let\OLDthebibliography\thebibliography
\renewcommand\thebibliography[1]{\OLDthebibliography{#1}\setlength{\parskip}{0pt}\setlength{\itemsep}{0pt plus 0.3ex}}
\usepackage{algorithm}
\usepackage{algpseudocode}
\usepackage{float}
\usepackage{tikz}

\usetikzlibrary{
	arrows.meta,
	positioning,
	shapes.geometric,
	calc,
	fit
}

\everymath{\displaystyle}
\allowdisplaybreaks

\newtheorem{theorem}{Theorem}[section]
\newtheorem{remark}[theorem]{Remark}
\newtheorem{example}[theorem]{Example}

\newtheorem{lemma}[theorem]{Lemma}
\newtheorem{corollary}[theorem]{Corollary}

\newtheorem{proposition}[theorem]{Proposition}

\newtheorem{question}{Question}
\renewenvironment{proof}{{\noindent\bfseries Proof. }}{\qed\par}
\begin{document}
	\title{Domination polynomial of co-maximal graphs of integer modulo ring}
	\author{Bilal Ahmad Rather\\
		{\small\em School of Mathematics and Statistics, Shandong University of Technology, Zibo 255049, China}\\ 
		\texttt{bilalahmadrr@gmail.com}}
	\date{}
	\pagestyle{myheadings}\markboth{Bilal Ahmad Rather}{Domination polynomial of co-maximal graphs of integer modulo ring}
	\maketitle
	\begin{abstract}
		We investigate the domination polynomial of the co-maximal graph $\Gamma(\mathbb{Z}_n)$ related to the ring of integers modulo $n$. Such a polynomial encodes domination sets of all possible sizes, a NP-hard problem for  general graph. We give exact formulas for prime powers, integers with two distinct prime divisors, and products of three distinct primes. A prime-support decomposition gives a finite explicit formula for arbitrary $n$, including repeated prime factors. Further results determine the highest coefficients, the precise threshold beyond which every vertex subset dominates, and the number of non-dominating subsets.  We discuss unimodality and log-concavity, and derive justified bounds for nonzero domination roots using the Enestr\"om--Kakeya theorem. 
	\end{abstract}

	
	\section{Introduction}
	We consider finite, simple and undirected graphs. A graph $G$ has vertex set $V(G)$ and edge set $E(G)$; its order and size are $|V(G)|$ and $|E(G)|$, respectively. The notation $u\sim v$ means that distinct vertices $u$ and $v$ are adjacent. A vertex of degree $0$ is isolated, whereas a vertex of degree $1$ is pendant. The degree $d_v(G)$ of $v$ is the number of vertices adjacent to $v$. The union $G_1\cup G_2$ of graphs on disjoint vertex sets has the edges of both graphs, and their \emph{join} $G_1\vee G_2$ is obtained by adding every edge between $V(G_1)$ and $V(G_2)$.
	
	A nonempty subset $S\subseteq V(G)$ is a \emph{dominating set} if each vertex in $V(G)\setminus S$ has a neighbour in $S$. The \emph{domination number} $\gamma(G)$ is the least cardinality of such a set. Domination has a substantial literature, including the treatment in \cite{haynes}. An \emph{independent set} contains no two adjacent vertices; its maximum size is $\alpha(G)$. An independent dominating set is both independent and dominating, and the least possible size is denoted by $\gamma_i(G)$. In particular, $\gamma(G)\leq\gamma_i(G)\leq\alpha(G)$. A clique is a set of pairwise adjacent vertices; $K_n$ is the complete graph of order $n$, and $\overline K_n$ is its complement.
	
	For $0\leq k\leq |V(G)|$, let $d(G,k)$ denote the number of dominating sets of cardinality $k$. The \emph{domination polynomial} is
	$$D(G,x)=\sum_{k=\gamma(G)}^{|V(G)|}d(G,k)x^k.$$
	Its degree is $|V(G)|$ and its leading coefficient is $1$. Any complex zero of $D(G,x)$ is a \emph{domination root}. Computing the domination number is an NP-hard problem in general, see \cite{garay}. The polynomial additionally records the entire distribution of dominating sets. This information is useful in combinatorial enumeration and in studying network monitoring, facility location, and related covering problems.
	
	Complete multipartite graphs have been studied through their domination polynomials \cite{anthony}. Alikhani, Brown and Jahari investigated the domination polynomials and roots of friendship graphs \cite{alikhani}; Akbari, Alikhani and Peng considered graphs characterized by their domination polynomials \cite{akbari}. Domination polynomials arising from zero-divisor graphs of finite commutative rings appear in \cite{alikhanizero,gursoy}. Further work on domination roots and polynomial properties can be found in \cite{alikhani1,alikhani2,mertens,BrownTufts2014}. Related graph polynomials include independence polynomials \cite{mandrescu,alavi,wang,brown}, independent domination polynomials \cite{anwar,dod,goddard,jahari}, chromatic polynomials \cite{read,birkhoff,tutte}, and generalizations of matching polynomials \cite{gutman}. Further aspects of algebraic graphs are described in \cite{Grau}, and several studies of domination-related polynomials for algebraically defined graphs appear in \cite{bilaldml,bilalsc,bilaltcs,bilaljac,bilaljcmcc}.

	Section \ref{section 2} recalls the co-maximal graph, derives the prime-power and two-prime formulas, and gives an exact expression for the three-prime case. Section \ref{section general} introduces the prime-support formula and several consequences concerning coefficients and total domination-set counts. Section \ref{section additional} establishes further structural and coefficient formulas. Section \ref{section 3} examines log-concavity and unimodality, presents corrected examples, and establishes annular bounds for nonzero domination roots.
	
	\section{Domination polynomial of co-maximal graph of $\mathbb Z_n$}\label{section 2}
	The relation between finite rings and graphs makes it possible to examine ring-theoretic structure through graph invariants. Among the familiar constructions are zero-divisor and co-zero-divisor graphs, unit graphs, ideal graphs and Cayley graphs of rings. Sharma and Bhatwadekar \cite{sharma} introduced a graph whose adjacency is defined by co-maximal ideals. For a commutative ring $R$ with identity $1\ne0$, the \emph{co-maximal graph} $\Gamma(R)$ has vertex set $R$; distinct $a,b\in R$ are adjacent precisely when $aR+bR=R$. The co-maximal graph associated with $\mathbb Z_n$ admits a particularly useful partition into arithmetic classes. The elements that are units are adjacent to every other vertex, while adjacency between nonunits depends only on the prime divisors shared by the corresponding residue classes. The structural facts on co-maximal graphs that we use are recalled from \cite{afkhami,banerjeeMatrices,maimani}.  The structural investigations of this construction include \cite{esmaili,sinha,samei,maimani}. Homological aspects of the associated edge rings were considered in \cite{bilalca}, while spectral studies can be found in \cite{afkhami,bilalijpam,bilaljaa}.
	
	Write $n=\prod_{i=1}^{r}p_i^{a_i}$, where the primes $p_i$ are distinct and $a_i$ are positive integers. The number of positive divisors of $n$, denoted by $\tau(n)$, satisfies
	\begin{equation}\label{divisors of n}
		\tau(n)=\prod_{i=1}^{r}(a_i+1).
	\end{equation}
	Euler's totient function $\phi(n)$ counts the residues relatively prime to $n$. We use $\sum_{d\mid n}\phi(d)=n$, $\phi(uv)=\phi(u)\phi(v)$ when $\gcd(u,v)=1$, and $\sum_{j=1}^{\ell}\phi(p^j)=p^\ell-1$. Here $\gcd(u,v)$ denotes the greatest common divisor of $u$ and $v$. The \emph{radical} $\operatorname{rad}(n)=p_1p_2\cdots p_r$ is the product of the distinct prime divisors of $n$.
	
	A proper divisor $d$ of $n$ satisfies $1<d<n$. Let $d_1,\ldots,d_t$ be the distinct proper divisors of $n$ and let $G_n$ denote their \emph{proper-divisor graph}, in which $d_i$ and $d_j$ are adjacent if and only if $\gcd(d_i,d_j)=1$. Thus $t=\tau(n)-2$. For each proper divisor $d$ define
	$$A_d=\{z\in\mathbb Z_n:\gcd(z,n)=d\},$$
	where a residue is represented by any integer and $\gcd(0,n)=n$. These classes are pairwise disjoint and $|A_d|=\phi(n/d)$, as recalled in \cite{my}. The unit group is
	$U(\mathbb Z_n)=\{z\in\mathbb Z_n:\gcd(z,n)=1\},$ with $|U(\mathbb Z_n)|=\phi(n).$
	Consequently,
	$$V(\Gamma(\mathbb Z_n))=U(\mathbb Z_n)\,\dot\cup\,\{0\}\,\dot\cup\,\bigsqcup_{d\mid n,\ 1<d<n}A_d,$$
	where $\dot\cup$ denotes disjoint union.

	The subsequent lemma \cite{afkhami,banerjeeMatrices} states that the induced subgraphs $ \Gamma(A_{d_i}) $ of $ \Gamma(\mathbb{Z}_{n}) $ are complements of certain cliques.
	\begin{lemma}[\cite{afkhami,banerjeeMatrices}]\label{adjacency relations of proper divisor graph}
		For proper divisors $d,e$ of $n$, the following assertions hold.
		\begin{itemize}
			\item[\bf (i)] If $d\ne e$, each vertex of $A_d$ is adjacent to each vertex of $A_e$ if and only if $\gcd(d,e)=1$.
			\item[\bf (ii)] Adjacency between two distinct divisor classes is either complete or absent.
			\item[\bf (iii)] Each $A_d$ induces an empty graph $\overline K_{\phi(n/d)}$.
		\end{itemize}
	\end{lemma}

	The \emph{generalized join} $F[H_1,\ldots,H_t]$ is the graph obtained by replacing each vertex $i$ of a graph $F$ by $H_i$ and joining every vertex of $H_i$ to every vertex of $H_j$ whenever $ij$ is an edge of $F$. The following lemma states that $ \Gamma(\mathbb{Z}_{n}) $ is the join of certain complete graphs and null graphs.
	\begin{lemma}[\cite{afkhami,banerjeeMatrices}]\label{joined union of co-maximal graph}
		For $n\ge2$, each $\Gamma(\mathbb Z_n)[A_d]\cong\overline K_{\phi(n/d)}$ and
		$$\Gamma(\mathbb Z_n)\cong K_{\phi(n)}\vee(K_1\cup G_2),$$
		where
		$$G_2=G_n\big[\overline K_{\phi(n/d_1)},\ldots,\overline K_{\phi(n/d_t)}\big].$$
		For a prime $n$, $G_2$ is understood to be the empty graph on no vertices.
	\end{lemma}

	We write $H_n=\Gamma(\mathbb Z_n)[\mathbb Z_n\setminus U(\mathbb Z_n)]$ for the subgraph induced by \emph{all nonunits}; thus $H_n=K_1\cup G_2$. For later use recall the standard identities \cite{alikhani}:
	\begin{align}
		D(G_1\cup G_2,x)&=D(G_1,x)D(G_2,x),\label{eq union}\\
		D(G_1\vee G_2,x)&=((1+x)^{|V(G_1)|}-1)((1+x)^{|V(G_2)|}-1)+D(G_1,x)+D(G_2,x).\label{eq join}
	\end{align}
	Also $D(K_s,x)=(1+x)^s-1$ and $D(\overline K_s,x)=x^s$. These formulas count separately the dominating sets meeting both joined parts and those confined to one part.
	
	\begin{theorem}\label{dom of G(zn)}
		Let $n\ge2$, $u=\phi(n)$, and $c=n/\operatorname{rad}(n)$. Then
		\begin{equation}\label{eq master}
			D(\Gamma(\mathbb Z_n),x)=(1+x)^n-(1+x)^{n-u}+D(H_n,x).
		\end{equation}
		Moreover, if $H_n'=H_n[\{z\in\mathbb Z_n:\operatorname{rad}(n)\nmid z,\ \gcd(z,n)>1\}]$, then
		\begin{equation}\label{eq master isolate}
			H_n\cong\overline K_c\cup H_n',\qquad D(H_n,x)=x^cD(H_n',x).
		\end{equation}
		The convention $D(\varnothing,x)=1$ is used for a graph with no vertices.
	\end{theorem}
	\begin{proof}
		By Lemma \ref{joined union of co-maximal graph}, $\Gamma(\mathbb Z_n)\cong K_u\vee H_n$ and $|V(H_n)|=n-u$. Equation \eqref{eq join} gives \eqref{eq master}. Since exactly $c$ residues are divisible by every distinct prime factor of $n$; no two of them, nor any one of them and any other nonunit, are adjacent. They form $\overline K_c$ inside $H_n$, giving \eqref{eq master isolate} by \eqref{eq union}.
	\end{proof}
	
	In the square-free case $c=1$, formula \eqref{eq master isolate} reads $D(H_n,x)=xD(G_2,x)$ since zero is the only isolated vertex. For repeated prime factors, all multiples of $\operatorname{rad}(n)$, including zero, form the isolated block. This distinction is essential when deriving explicit formulas.
	\begin{theorem}\label{dom of G(zn)'}
		Let $ G\cong \Gamma(\mathbb{Z}_n) $ be the co-maximal graph of order $ n. $ Then the following holds.
		\begin{itemize}
			\item [\bf (i)] If $ n=\prod_{i=1}^{t}p_{i} $ is product of $k\geq 2$ distinct primes, then the domination polynomial of $ G $ is
			\[ D(G,x)=(1+x)^{n}-(1+x)^{n-\phi(n)}+xD(G_{2},x), \]
			where $ D(G_{2},x) $ is the domination polynomial $ G_{2} $.
			\item[\bf (ii)] If  the canonical representation of $n$ is $ p_{1}^{n_{1}}p_{2}^{n_{2}}\dots p_{k}^{n_{k}}, $ then the domination polynomial of $ G $ is
			\[ D(G,x)=(1+x)^{\phi(n)}-(1+x)^{n-\phi(n)}+x^{\overline{n}}D(G_{2},x), \]
			where $\overline{n}=\frac{n}{\prod_{i=1}^{k}p_{i}}$ and $ D(G_{2},x) $ is the domination polynomial of the connected component $ G_{2}$ in $G.$
		\end{itemize}
	\end{theorem}
	\begin{proof}
		For a finite commutative ring $ \mathbb{Z}_{n} $ with its co-maximal graph $G\cong \Gamma(\mathbb{Z}_{n}) $,  the  relatively prime $\phi(n)$ elements with $n$  in graph $ \mathbb{Z}_{n} $ form a clique of size $\phi(n)$ as such elements are adjacent to every other vertex of $ \Gamma(\mathbb{Z}_{n}) $ since they generate all elements in $\mathbb{Z}_{n}$. Thus, with application of Lemma \ref{joined union of co-maximal graph}, the graphs $ \Gamma(\mathbb{Z}_{n}) $ can be put as
		\begin{align*}
			G\cong	\Gamma(\mathbb{Z}_{n})=K_{\phi(n)}\vee \left(K_{1}\cup G_{2} \right).
		\end{align*}
		Now with $ n=\prod_{i=1}^{k}p_{i} $, the graph $ G_{n} $ is connected as $p_{i}$ is adjacent to all vertices in $G_{n}$ except multiplies of $p_{i}$ and another vertex $p_{j}$ is adjacent to all vertices except multiplies of $p_{j}$, and in return $p_{i}$ is adjacent to $p_{j}$ for $i\neq j$, thereby $p_{i}p_{j}$ is adjacent to some $p_{k}$ with $k\notin \{i,j\}.$ Thus, $ G_{n} $ is connected graph and so is $G_{2}$. it is obvious that each vertex in $ K_{\phi(n)} $ is adjacent to all other vertices in $ G$ and it follows that each vertex singleton $\{v\}\in V(K_{\phi(n)})$ is a dominating vertex. Likewise, with $\{0\}\cup \{v\} $ is another dominating set in $G.$ Thus, with these observation and application of \eqref{eq union} and \eqref{eq join}, we have the following generating formulae for the domination polynomial of $G$
		\begin{align*}
			D(G,x)&= \big((1+x)^{\phi(n)}-1\big)\big((1+x)^{n-\phi(n)}-1\big)+(1+x)^{\phi(n)}-1+xD(G_{2},x),\\
			&=(1+x)^{n}-(1+x)^{n-\phi(n)}+xD(G_{2},x),
		\end{align*}
		where $ D(G_{2},x) $ is the domination polynomial of $ G_{2}. $\\
		For $ n=\prod_{i=1}^{k}p_{i}^{n_{i}} $ with primes $ p_{i} $ and $ n_{i} $'s as positive integers with $ n_{i}\geq 2, $ for at least one $i.$
		In $ G_{2} $, the collection of vertices/elements generated by the ideal $ \langle p_{1}p_{2}\dots p_{k } \rangle\setminus\{0\} $ are never relatively prime with other vertices of $ G_{2} $. It follows that all such vertices  form an isolated set of order $ \frac{n}{\prod_{i=1}^{k}p_{i}}-1 $. Together with $\{0\}$, total number of isolated vertices in  $G_{2}$ are $ \prod_{i=1}^{k}p_{i}^{n_{i}-1} $ isolated vertices in $ K_{1}\cup G_{2} $. Also, such vertices are always part  of domination set other than the vertices of $ K_{\phi(n)}. $ Thus, by \eqref{eq union} and \eqref{eq join}, we have get the following parallel formulae for the domination polynomial of $ G $
		\begin{align*}
			D(G,x)&= \big((1+x)^{\phi(n)}-1\big)\big((1+x)^{n-\phi(n)}-1\big)+(1+x)^{\phi(n)}-1+x^{\overline{n}}D(G_{2},x),\\
			&=(1+x)^{n}-(1+x)^{n-\phi(n)}+x^{\overline{n}}D(G_{2},x),
		\end{align*}
		where $\overline{n}=\prod_{i=1}^{k}p_{i}^{n_{i}-1}$ and $ D(G_{2},x) $ is the domination polynomial of subgraph $ G_{2} $ in $G.$
	\end{proof}

	\medskip
	
	From \ref{dom of G(zn)}, we observe that the domination polynomial of $G\cong \Gamma(\mathbb{Z}_{n})$ depends on the subgraph $G_{2}$ of $G.$ But, from the construction of $G_{n}$ as $n$ increase, so does its divisors and thereby structure of $G_{2}$ becomes complex and hence it becomes quit technical to write the explicit expression for the domination polynomial of $\Gamma(\mathbb{Z}_{n}) $ for general value of $n.$ However, we have the following consequence from Theorem \ref{dom of G(zn)} presenting exact expression for domination polynomial of $\Gamma(\mathbb{Z}_{n})$ for special values of $n.$
	\begin{theorem}\label{domination poly of zn}
		Let $n\ge2$. The following identities hold.
		\begin{itemize}
			\item[\bf (i)] If $n=p$ is prime, then
			$$D(\Gamma(\mathbb Z_p),x)=(1+x)^p-1=\sum_{i=1}^{p}\binom pi x^i.$$
			\item[\bf (ii)] If $n=p^m$, where $p$ is prime and $m\ge2$, and $a=p^{m-1}$, then
			\begin{equation}\label{eq pp}
				D(\Gamma(\mathbb Z_{p^m}),x)=(1+x)^{p^m}-(1+x)^a+x^a.
			\end{equation}
			Equivalently, this polynomial is $((1+x)^{a(p-1)}-1)(1+x)^a+x^a$.
		\end{itemize}
	\end{theorem}
	\begin{proof}
		For $n=p$, every nonzero residue is a unit and $\Gamma(\mathbb Z_p)\cong K_p$. For $n=p^m$ the nonunits are precisely the multiples of $p$, which form an independent set of cardinality $a$. Hence
		$$\Gamma(\mathbb Z_{p^m})\cong K_{p^m-a}\vee\overline K_a.$$
		Now from \eqref{eq join}, we obtain \eqref{eq pp}.
	\end{proof}
	
	\begin{example}\label{ex 32}
		For $n=2^5=32$, we have $a=16$ and
		$$D(\Gamma(\mathbb Z_{32}),x)=(1+x)^{32}-(1+x)^{16}+x^{16}.$$
		The first terms are
		$$16x+376x^2+4400x^3+34140x^4+197008x^5+898184x^6+\cdots+x^{32}.$$
		In particular, $[x]D=16$, and the degree is $32$, as required by the graph order. Here and below $[x^k]f(x)$ denotes the coefficient of $x^k$ in $f(x)$.
	\end{example}
	
	\begin{theorem}\label{domination poly of zn n=pq}
		Let $n=p^{a}q^{b}$ for distinct primes $p<q$ and integers $a,b\ge1$. Put
		$c=\frac{n}{pq}, A=c(q-1), B=c(p-1),$ and $u=\phi(n).$
		Then
		\begin{equation}\label{eq pq full}
			\begin{split}
				D(\Gamma(\mathbb Z_n),x)=&(1+x)^n-(1+x)^{n-u} +x^c\big[\big((1+x)^A-1\big)\big((1+x)^B-1\big)+x^A+x^B\big].
			\end{split}
		\end{equation}
		In particular, when $a=b=1$,
		\begin{equation}\label{eq pq squarefree}
			\begin{split}
				D(\Gamma(\mathbb Z_{pq}),x)=&(1+x)^{pq}-(1+x)^{p+q-1} +x\big((1+x)^{p-1}-1\big)\big((1+x)^{q-1}-1\big)+x^p+x^q.
			\end{split}
		\end{equation}
	\end{theorem}
	\begin{proof}
		A nonunit of $\mathbb Z_n$ is either divisible by $p$ but not $q$, divisible by $q$ but not $p$, or divisible by both. The three classes have sizes $A,B,c$, respectively. The first two classes induce the complete bipartite graph $K_{A,B}$, and the last is an isolated block. Thus
		$H_n\cong\overline K_c\cup K_{A,B}.$ 	The domination polynomial of the complete bipartite graph is
		$$D(K_{A,B},x)=\big((1+x)^A-1\big)\big((1+x)^B-1\big)+x^A+x^B,$$
		since a dominating set either meets both parts or comprises one complete part. Now, form Equations \eqref{eq master} and \eqref{eq union}, we obtain \eqref{eq pq full}. For $a=b=1$, $c=1$, $A=q-1$, $B=p-1$, and $u=(p-1)(q-1)$,  \eqref{eq pq squarefree} follows.
	\end{proof}
	
	\begin{remark}
		In \eqref{eq pq full} the factor $x^c$ multiplies the \emph{entire} domination polynomial of $K_{A,B}$, including $x^A+x^B$. The isolated block must belong to every dominating set that avoids the unit clique. This also explains why the exponents of the two final monomials are $c+A$ and $c+B$.
	\end{remark}
	
	\begin{example}\label{ex 15}
		For $n=3\cdot5=15$, we have $u=8$, $c=1$, $A=4$, $B=2$. From \eqref{eq pq squarefree}, we have
		\begin{align*}
			D(\Gamma(\mathbb Z_{15}),x)={}&8x+84x^2+429x^3+1346x^4+2997x^5+5004x^6 +6435x^7+6435x^8\\
			&+5005x^9+3003x^{10} +1365x^{11}+455x^{12}+105x^{13}+15x^{14}+x^{15}.
		\end{align*}
	\end{example}
	
	Next, we  consider the product $n=pqr$ of three distinct primes. The proper-divisor graph consists of the six nonzero, nonunit blocks associated with $p,q,r,pq,pr,qr$, together with the isolated zero residue. Figure \ref{fig 1} gives the block-diagram description of $\Gamma(\mathbb{Z}_{pqr})$, where solid lines signify complete adjacency between blocks, dashed lines signify adjacency to every unit, and each displayed divisor block is independent.
	\begin{figure}[H]
		\centering
		\begin{tikzpicture}[
			font=\small,
			block/.style={draw,rounded corners,align=center,minimum width=30mm,minimum height=10mm,inner sep=2mm},
			clique/.style={block,double,very thick},
			iso/.style={circle,draw,minimum size=7mm,inner sep=0mm},
			complete/.style={very thick},
			joinstyle/.style={dashed,very thick}]
			\node[clique] (U) {$U(\mathbb{Z}_{pqr})$\\[1mm] $K_{\phi(pqr)}$};
			\node[iso,left=18mm of U] (z) {$0$};
			\node[block,below left=20mm and 18mm of U] (Ap) {$A_p$\\[1mm] $\overline{K}_{\phi(qr)}$};
			\node[block,below right=20mm and 18mm of U] (Aq) {$A_q$\\[1mm] $\overline{K}_{\phi(pr)}$};
			\node[block,below=30mm of U,xshift=-8mm] (Ar) {$A_r$\\[1mm] $\overline{K}_{\phi(pq)}$};
			\node[block,below=10mm of Ap,xshift=2mm] (Aqr) {$A_{qr}$\\[1mm] $\overline{K}_{\phi(p)}$};
			\node[block,below=10mm of Aq,xshift=-2mm] (Apr) {$A_{pr}$\\[1mm] $\overline{K}_{\phi(q)}$};
			\node[block,below=10mm of Ar,xshift=22mm] (Apq) {$A_{pq}$\\[1mm] $\overline{K}_{\phi(r)}$};
			\draw[complete] (Ap)--(Aq);\draw[complete] (Ap)--(Ar);\draw[complete] (Aq)--(Ar);
			\draw[complete] (Ap)--(Aqr);\draw[complete] (Aq)--(Apr);\draw[complete] (Ar)--(Apq);
			\draw[joinstyle] (U)--(z);\draw[joinstyle] (U)--(Ap);\draw[joinstyle] (U)--(Aq);
			\draw[joinstyle] (U)--(Ar);\draw[joinstyle] (U)--(Aqr);\draw[joinstyle] (U)--(Apr);\draw[joinstyle] (U)--(Apq);
			\node[draw,dashed,inner sep=6mm,fit=(Ap)(Aq)(Ar)(Aqr)(Apr)(Apq)] (G2box) {};
			\node[anchor=south,yshift=-10mm] at (G2box.north) {$G_2$};
		\end{tikzpicture}
		\caption{Block diagram of comaximal graph $\Gamma(\mathbb{Z}_{pqr})$.}
		\label{fig 1}
	\end{figure}
	
	Write the three central blocks $C_1=A_p$, $C_2=A_q$, $C_3=A_r$, and their associated pendant blocks $L_1=A_{qr}$, $L_2=A_{pr}$, $L_3=A_{pq}$. Set $a_1=\phi(qr),~ a_2=\phi(pr),~ a_3=\phi(pq),~
	b_1=\phi(p),~ b_2=\phi(q),$ and $b_3=\phi(r).$
	The central blocks are pairwise completely joined and $L_i$ is adjacent only to $C_i$ within $G_2$. A selection of a nonempty subset of each central block is dominating, but other possibilities must be counted without including undominated vertices in an unselected central block.
	
	\begin{theorem}\label{thm triple}
		For $n=pqr$ with distinct primes, let
		$F_i(x)=\big((1+x)^{a_i}-1\big)(1+x)^{b_i},~
		E_i(x)=\big((1+x)^{a_i}-1\big)\big((1+x)^{b_i}-1\big)+x^{a_i}.$
		Then
		\begin{equation}\label{eq triple}
			\begin{split}
				D(G_2,x)={}&x^{b_1+b_2+b_3}
				+\sum_{i=1}^3E_i(x)\prod_{j\ne i}x^{b_j}+\sum_{\substack{I\subseteq\{1,2,3\}\\|I|\ge2}}
				\left(\prod_{i\in I}F_i(x)\right)\left(\prod_{j\notin I}x^{b_j}\right).
			\end{split}
		\end{equation}
		Consequently,
		$$D(\Gamma(\mathbb Z_{pqr}),x)=(1+x)^{pqr}-(1+x)^{pqr-\phi(pqr)}+xD(G_2,x).$$
	\end{theorem}

	\begin{proof}
		We recall that the central blocks $C_1,C_2,C_3$ are pairwise
		completely joined, whereas each pendant block $L_i$ is
		adjacent only to $C_i$. Let $S$ be a dominating set of
		$G_2$ and let
		$$
		I=\{i\in\{1,2,3\}:S\cap C_i\neq\varnothing\}.
		$$
		If $I=\varnothing$, every vertex of each $L_i$ must
		belong to $S$, since otherwise it has no neighbour in
		$S$. Conversely, these vertices dominate all central
		blocks, and gives $x^{b_1+b_2+b_3}$.
		
		Suppose $I=\{i\}$. The central selection dominates
		$C_j$ for $j\neq i$, while every $L_j$ with $j\neq i$
		must be selected in full. Moreover, the unselected
		vertices of $C_i$ can only be dominated by a nonempty
		selection from $L_i$. Hence either both $C_i$ and
		$L_i$ meet $S$, or $C_i\subseteq S$ and
		$S\cap L_i=\varnothing$. These disjoint possibilities
		contribute $E_i(x)$, giving
		$$
		\sum_{i=1}^{3}E_i(x)\prod_{j\neq i}x^{b_j}.
		$$
		
		Finally, if $|I|\geq2$, every central block is
		dominated by the selection in another central block.
		For each $i\in I$, a nonempty subset of $C_i$ may
		be chosen together with an arbitrary subset of $L_i$,
		contributing $F_i(x)$. For $j\notin I$, the entire
		block $L_j$ must be selected. Thus these choices
		contribute
		$$
		\sum_{\substack{I\subseteq\{1,2,3\}\\|I|\geq2}}
		\left(\prod_{i\in I}F_i(x)\right)
		\left(\prod_{j\notin I}x^{b_j}\right).
		$$
		The three cases are mutually exclusive and exhaustive,
		which establishes \eqref{eq triple}. The asserted
		formula for $D(\Gamma(\mathbb Z_{pqr}),x)$ follows
		from the decomposition
		$
		\Gamma(\mathbb Z_{pqr})
		\cong K_{\phi(pqr)}\vee(K_1\cup G_2)
		$
		and the standard join and disjoint-union identities
		for domination polynomials.
	\end{proof}

	\begin{example}\label{ex 30}
		For $(p,q,r)=(2,3,5)$, we have $n=30$, with central-block
		sizes $(a_1,a_2,a_3)=(8,4,2)$ and pendant-block sizes
		$(b_1,b_2,b_3)=(1,2,4)$. Substituting these values into
		\eqref{eq triple} and simplifying, we obtain
		\begin{equation*}\label{eq G2 30}
			\begin{aligned}
				D(G_2,x)={}&72x^3+880x^4+5199x^5+19742x^6 +54190x^7+113896x^8+189319x^9\\
				&+253574x^{10}+276576x^{11}+246851x^{12} +180348x^{13}+107430x^{14}\\
				&+51722x^{15}+19833x^{16}+5919x^{17}+1326x^{18} +210x^{19}+21x^{20}+x^{21}.
			\end{aligned}
		\end{equation*}
	\end{example}

	The same prime-support method used below applies to $p^a q^b r^c$ without assuming square-freeness, and to integers with any number of distinct prime divisors.
	
	\section{Prime-support formulas and further results}\label{section general}
	The divisor-class construction can be compressed considerably: adjacency of two nonunits depends only on the \emph{sets} of primes dividing their representatives, not on their precise greatest common divisors with $n$. This leads to a uniform expression for arbitrary $n$.
	
	For $n=\prod_{i=1}^r p_i^{a_i}$ and a subset $S\subseteq[r]=\{1,\ldots,r\}$, let
	$$B_S=\{z\in\mathbb Z_n:p_i\mid z\text{ exactly for those }i\in S\}.$$
	Define $w_S=|B_S|$. The Chinese remainder theorem gives
	\begin{equation}\label{eq weights}
		w_S=\prod_{i\in S}p_i^{a_i-1}\prod_{j\notin S}\phi(p_j^{a_j}),\quad
		w_{\varnothing}=\phi(n),\quad w_{[r]}=\frac{n}{\operatorname{rad}(n)}=c.
	\end{equation}
	Each $B_S$ with $S\ne\varnothing$ is an independent set. For two distinct nonempty supports $S,T$, every pair of vertices from $B_S$ and $B_T$ is adjacent precisely when $S\cap T=\varnothing$.
	
	Let $\mathcal S=2^{[r]}\setminus\{\varnothing\}$ and let $\mathcal A$ be the graph with vertex set $\mathcal S$ and adjacency given by disjointness. If $J\subseteq\mathcal S$, let
	$$N_{\mathcal A}(S)\cap J=\{T\in J:T\cap S=\varnothing\}.$$
	We call $J$ a dominating support family when for each $S\in\mathcal S\setminus J$ this set is nonempty.
	
	\begin{theorem}\label{thm support}
		For every integer $n\ge2$,
		\begin{equation}\label{eq support}
			D(H_n,x)=\sum_{\substack{J\subseteq\mathcal S\\
					\forall S\in\mathcal S\setminus J:\ N_{\mathcal A}(S)\cap J\ne\varnothing}}
			\prod_{S\in J}
			\begin{cases}
				(1+x)^{w_S}-1,&N_{\mathcal A}(S)\cap J\ne\varnothing,\\
				x^{w_S},&N_{\mathcal A}(S)\cap J=\varnothing.
			\end{cases}
		\end{equation}
		In combination with \eqref{eq master}, this is an exact finite expression for $D(\Gamma(\mathbb Z_n),x)$ for every prime factorization of $n$. It contains at most $2^{2^r-1}$ support-family summands and no summation over individual residues.
	\end{theorem}
	
	\begin{proof}
		Let $X$ be a dominating set of $H_n$ and define
		$$
		J=\{S\in\mathcal S:X\cap B_S\neq\varnothing\}.
		$$
		By the construction of $H_n$, each block $B_S$ is
		independent, and two distinct blocks $B_S,B_T$ are
		completely joined precisely when $S\cap T=\varnothing$.
		Thus, if $S\notin J$, every vertex of $B_S$ must be
		dominated from a selected neighbouring block, which		is equivalent to
		$$
		N_{\mathcal A}(S)\cap J\neq\varnothing,
		$$
		and the restriction in \eqref{eq support} follows.
		
		Now for a fixed admissible support family $J$, if
		$S\in J$ has a neighbour in $J$, then every unselected
		vertex of $B_S$ is dominated externally. Hence any
		nonempty subset of $B_S$ may be chosen, contributing
		$(1+x)^{w_S}-1$. Otherwise, no selected vertex outside
		$B_S$ is adjacent to it. Since $B_S$ is independent,
		the entire block must belong to $X$, contributing
		$x^{w_S}$.
		
		These conditions are sufficient to dominate every
		block and, for fixed $J$, the selections are
		independent across blocks. Therefore their generating function
		is the product in \eqref{eq support}.
		Moreover, each dominating set determines exactly one
		support family $J$, so the sum counts every dominating
		set once. Finally, $|\mathcal S|\leq 2^r-1$, giving at
		most $2^{2^r-1}$ summands. By combining this expression
		with \eqref{eq master}, the proof follows.
	\end{proof}

	 Formula \eqref{eq support} gives a direct symbolic algorithm. Form the $2^r-1$ positive weights $w_S$, precompute disjointness of support pairs, enumerate each family $J$, discard it if an omitted support is not adjacent to a member of $J$, and multiply the stated factors for the remaining families. For fixed $r$, the computation is polynomial in $n$, even when $n$ has large prime-power factors. The formula is also practical as a compact symbolic expression without expanding every coefficient.

	The support decomposition in Theorem~\ref{thm support}
	provides a finite procedure for computing the domination
	polynomial without examining individual vertex subsets.
	The following algorithm implements this decomposition
	using the adjacency relations between prime supports.
	
	\begin{algorithm}[H]
		\caption{Domination polynomial through prime-support decomposition}
		\label{alg support}
		\begin{algorithmic}[1]
			\Require $n=\prod_{i=1}^{r}p_i^{e_i}$, where
			$p_1,\ldots,p_r$ are distinct primes and $e_i\geq1$.
			\Ensure $D(\Gamma(\mathbb Z_n),x)$.
			
			\State Set $\mathcal S=2^{\{1,\ldots,r\}}
			\setminus\{\varnothing\}$ and $M=|\mathcal S|=2^r-1$.
			
			\ForAll{$S\in\mathcal S$}
			\State Compute
			$w_S=\displaystyle\prod_{i=1}^{r}p_i^{e_i-1}
			\prod_{i\notin S}(p_i-1)$.
			\State Set $f_S(x)=(1+x)^{w_S}-1$.
			\EndFor
			
			\ForAll{$S,T\in\mathcal S$}
			\State Set $A_{S,T}=1$ if
			$S\cap T=\varnothing$, and $A_{S,T}=0$
			otherwise.
			\EndFor
			
			\State Initialize $P(x)\gets0$.
			
			\ForAll{$J\subseteq\mathcal S$}
			\State Set $\mathrm{valid}\gets\mathrm{true}$.
			
			\ForAll{$S\in\mathcal S\setminus J$}
			\If{$\displaystyle\sum_{T\in J}A_{S,T}=0$}
			\State $\mathrm{valid}\gets\mathrm{false}$.
			\State \textbf{break}
			\EndIf
			\EndFor
			
			\If{$\mathrm{valid}=\mathrm{true}$}
			\State Set $Q(x)\gets1$.
			
			\ForAll{$S\in J$}
			\If{$\displaystyle\sum_{T\in J}A_{S,T}>0$}
			\State $Q(x)\gets Q(x)f_S(x)$.
			\Else
			\State $Q(x)\gets Q(x)x^{w_S}$.
			\EndIf
			\EndFor
			
			\State $P(x)\gets P(x)+Q(x)$.
			\EndIf
			\EndFor
			
			\State \Return
			$\displaystyle (1+x)^n-(1+x)^{n-\phi(n)}+P(x)$.
		\end{algorithmic}
	\end{algorithm}
	
	 Algorithm~\ref{alg support} enumerates admissible families
	of prime supports and computes their contributions to
	$D(H_n,x)$ through independent block selections.
	The flowchart in Figure~\ref{fig support flow} describes
	the admissibility test, polynomial accumulation, and
	recovery of $D(\Gamma(\mathbb Z_n),x)$.
	
	\begin{figure}[H]
		\centering
		\begin{tikzpicture}[
			>=Stealth,
			node distance=5mm and 14mm,
			font=\footnotesize,
			state/.style={
				rectangle,draw,rounded corners,
				align=center,text width=9.0cm,
				minimum height=8mm,inner sep=3pt
			},
			decision/.style={
				diamond,draw,aspect=3,
				align=center,text width=4.6cm,
				inner sep=2pt
			},
			terminal/.style={
				rectangle,draw,rounded corners,
				align=center,text width=4.0cm,
				minimum height=9mm,inner sep=3pt
			},
			line/.style={->,thick}
			]
			
			\node[state] (start)
			{Input $n=\prod_{i=1}^{r}p_i^{e_i}$};
			
			\node[state,below=of start] (prepare)
			{Construct $\mathcal S$, compute $w_S$ and
				$f_S(x)=(1+x)^{w_S}-1$, and form the
				disjointness matrix $A$};
			
			\node[state,below=of prepare] (initialize)
			{Set $M=2^r-1$, $P(x)=0$, and $j=0$};
			
			\node[decision,below=8mm of initialize] (check)
			{$j<2^M$?};
			
			\node[state,below=9mm of check] (family)
			{Decode $j$ as a support family
				$J\subseteq\mathcal S$};
			
			\node[decision,below=8mm of family] (valid)
			{$N_{\mathcal A}(S)\cap J\neq\varnothing$
				for every $S\notin J$?};
			
			\node[state,below=9mm of valid] (product)
			{Set $Q(x)=1$; multiply by $f_S(x)$ if
				$N_{\mathcal A}(S)\cap J\neq\varnothing$,
				and by $x^{w_S}$ otherwise, for each $S\in J$};
			
			\node[state,below=of product] (accumulate)
			{Update $P(x)\gets P(x)+Q(x)$};
			
			\node[state,below=of accumulate] (increment)
			{Update $j\gets j+1$};
			
			\node[terminal,right=26mm of check] (output)
			{Return\\
				$\displaystyle (1+x)^n
				-(1+x)^{n-\phi(n)}+P(x)$};
			
			\draw[line] (start) -- (prepare);
			\draw[line] (prepare) -- (initialize);
			\draw[line] (initialize) -- (check);
			
			\draw[line] (check) --
			node[right]{Yes} (family);
			
			\draw[line] (check) --
			node[above]{No} (output);
			
			\draw[line] (family) -- (valid);
			
			\draw[line] (valid) --
			node[right]{Yes} (product);
			
			\draw[line] (product) -- (accumulate);
			\draw[line] (accumulate) -- (increment);
			
			\draw[line] (valid.east) -- ++(13mm,0)
			|- node[pos=0.15,right]{No} (increment.east);
			
			\draw[line] (increment.west) -- ++(-12mm,0)
			|- (check.west);
			
		\end{tikzpicture}
		\caption{Flowchart for computing
			$D(\Gamma(\mathbb Z_n),x)$ using prime-support
			decomposition.}
		\label{fig support flow}
	\end{figure}
	
	The next result establishes an arithmetic complexity
	bound for Algorithm~\ref{alg support}. The estimate uses
	ordinary polynomial multiplication and does not assume
	any particular fast multiplication procedure.
	
	\begin{theorem}\label{thm support complexity}
		Let $n=\prod_{i=1}^{r}p_i^{e_i}\geq2$ and
		$M=2^r-1$. Algorithm~\ref{alg support} computes
		the domination polynomial of
		$\Gamma(\mathbb Z_n)$ using at most
		$
		O\left(
		M^2r+2^M(M^2+Mn^2)
		\right)
		$
		arithmetic operations. In particular, its complexity
		is bounded by
		\begin{equation}\label{eq support complexity}
			O\left(
			2^{2^r-1}
			\left(2^{2r}r+2^rn^2\right)
			\right).
		\end{equation}
		Consequently, for a fixed number $r$ of distinct
		prime divisors, the algorithm has polynomial
		arithmetic complexity in $n$.
	\end{theorem}
	
	\begin{proof}
		Let $M=|\mathcal S|=2^r-1$ and
		$h=n-\phi(n)=|V(H_n)|$. Then the algorithm first
		computes the $M$ block weights and the corresponding
		binomial polynomials. Each weight involves at most
		$r$ factors, so all weights require $O(Mr)$
		arithmetic operations. Since
		$$
		\sum_{S\in\mathcal S}w_S=h\leq n,
		$$
		the binomial coefficients can be generated
		recursively using at most $O(n)$ additional
		arithmetic operations. There are $M^2$ ordered pairs of supports.
		For testing disjointness, we require at most $r$
		comparisons per pair. Thus, construction of the
		adjacency matrix takes $O(M^2r)$ operations. The main loop examines exactly $2^M$ support
		families. For each $J\subseteq\mathcal S$, the
		admissibility condition requires checking whether
		every omitted support has a neighbour in $J$.
		Using the precomputed adjacency matrix, this
		test involves at most $M^2$ entries and hence
		requires $O(M^2)$ operations.
		
		For an admissible family $J$, the polynomial
		$Q(x)$ contains at most $|J|\leq M$ factors.
		Each factor has degree at most $h$, and every
		partial product also has degree at most $h$,
		since the block weights are additive over
		disjoint blocks. Naive convolution of two
		polynomials of degree at most $h$ requires
		$O(h^2)$ arithmetic operations. Consequently,
		forming $Q(x)$ takes at most $O(Mn^2)$ operations.
		The additional cost of accumulating $Q(x)$
		into $P(x)$ is $O(n)$ and is absorbed by this
		bound.
		
		Even if every support family passes the
		admissibility test, the total cost of the
		enumeration is therefore at most
		$$
		O\left(2^M(M^2+Mn^2)\right).
		$$
		The final evaluation of
		$(1+x)^n-(1+x)^{n-\phi(n)}+P(x)$ requires
		at most $O(n)$ further arithmetic operations.
		By combining these estimates, we obtain
		$$
		O\left(M^2r+2^M(M^2+Mn^2)\right).
		$$
		By substituting $M=2^r-1$, we obtain 		\eqref{eq support complexity}.
		
		Finally, if $r$ is fixed, both $M$ and $2^M$ are constants independent of $n$. The bound
		then reduces to $O(n^2)$ arithmetic operations.
	\end{proof}
	
	\begin{remark}\label{rem support complexity}
		The complexity estimate in
		Theorem~\ref{thm support complexity} counts
		arithmetic operations rather than bit operations.
		Since the coefficients of $D(H_n,x)$ count
		vertex subsets, they are at most $2^{n-\phi(n)}$
		and may require $O(n)$ bits. Thus, integer
		arithmetic contributes an additional bit cost.
		
		For fixed $r$, the algorithm is polynomial in
		the numerical value of $n$, but this does not
		imply polynomial complexity in the binary input
		length $O(\log n)$. When the polynomial is
		retained in factored form, the convolution
		stage can be avoided, making the support
		decomposition particularly useful for symbolic
		computation.
	\end{remark}

	The next result determines the domination number of the nonunit subgraph in all cases. This is distinct from the domination number of $\Gamma(\mathbb Z_n)$, which is $1$ because every unit is universal.
	\begin{theorem}\label{thm nonunit gamma}
		Let $n=\prod_{i=1}^{r}p_i^{a_i}$ and $c=n/\operatorname{rad}(n)$. Then
		\begin{equation}\label{eq gamma H}
			\gamma(H_n)=
			\begin{cases}
				c,&r=1,\\
				c+\min\{2,c(p_1-1),c(p_2-1)\},&r=2,\\
				c+r,&r\ge3.
			\end{cases}
		\end{equation}
		For $r\ge3$ the number of dominating sets of $H_n$ of minimum cardinality is
		\begin{equation}\label{eq mincount}
			d(H_n,c+r)=\prod_{i=1}^r w_{\{i\}}
			+\begin{cases}
				\displaystyle\prod_{j\ne i_0}w_{\{j\}},&c=1\text{ and }p_{i_0}=2,\\
				0,&\text{otherwise}.
			\end{cases}
		\end{equation}
		There is at most one index $i_0$ with $p_{i_0}=2$.
	\end{theorem}

	\begin{proof}
		The block $B_{[r]}$ consists of $c$ isolated vertices, so every dominating set of $H_n$ contains it entirely.
		For $r=1$, we get $\gamma(H_n)=c$. If $r=2$, then
		the remaining graph is
		$K_{c(p_2-1),c(p_1-1)}$, whose domination number is
		$\min\{2,c(p_1-1),c(p_2-1)\}$.
		
		Now suppose $r\geq3$. For each $i\in[r]$, let
		$C_i=B_{\{i\}}$ and $L_i=B_{[r]\setminus\{i\}}$.
		Every vertex of $L_i$ is adjacent only to vertices
		of $C_i$. Consequently, a dominating set must meet
		$C_i$ or contain $L_i$ entirely. Since the $r$
		pairs $C_i\cup L_i$ are mutually disjoint, at least
		$r$ vertices are required in addition to $B_{[r]}$.
		Conversely, one vertex from each $C_i$ dominates
		every nonempty proper support block, since every
		proper support omits some $i$. Hence
		$\gamma(H_n)=c+r$, proving \eqref{eq gamma H}.
		
		To attain this minimum, a dominating set must
		contain exactly one vertex from each $C_i\cup L_i$
		and no vertices from other blocks, apart from
		$B_{[r]}$. We choose one vertex from every $C_i$, and obtain 
		$
		\prod_{i=1}^{r}w_{\{i\}}
		$
		such sets. Alternatively, a vertex of $L_i$ may
		replace the selection from $C_i$ only when
		$|L_i|=1$, because otherwise the unselected
		vertices of $L_i$ would remain undominated.
		By \eqref{eq weights}, we have
		$
		|L_i|=w_{[r]\setminus\{i\}}=c(p_i-1),
		$
		so this occurs precisely when $c=1$ and $p_i=2$.
		There is at most one such index $i_0$.
		
		In this exceptional case, the selected vertex of
		$L_{i_0}$ dominates itself, while the vertices
		chosen from $C_j$, $j\neq i_0$, dominate all
		remaining blocks, including $C_{i_0}$.
		Their number is
		$
		\prod_{j\neq i_0}w_{\{j\}}.
		$
		No further minimum dominating sets are possible, as
		two or more pendant-block replacements would
		require two distinct prime divisors equal to $2$. By combining the two disjoint families, we obtain 
		\eqref{eq mincount}.
	\end{proof}
	
	The following consequence is immediate from above result.
	\begin{corollary}\label{cor lowcoeff}
		Let $g=\gamma(H_n)$ and $m=n-\phi(n)$. For $1\le k<g$,
		$$d(\Gamma(\mathbb Z_n),k)=\binom nk-\binom mk,$$
		with the convention $\binom mk=0$ for $k>m$. For $r\ge3$ the first correction to this binomial difference occurs at degree $c+r$, with coefficient given by \eqref{eq mincount}.
	\end{corollary}
	\begin{proof}
		Equation \eqref{eq master} implies
		$$d(\Gamma(\mathbb Z_n),k)=\binom nk-\binom mk+d(H_n,k).$$
		The last summand vanishes for $k<g$. For $k=g=c+r$, its value is given by Theorem \ref{thm nonunit gamma}.
	\end{proof}
	
	A second exact result describes the highest coefficients. It does not require a restriction on the number of distinct prime divisors.
	
	\begin{theorem}\label{thm top}
		For $n\ge2$ let $m=n-\phi(n)$ and $c=n/\operatorname{rad}(n)$. Then
		\begin{equation}\label{eq top coeff}
			d(\Gamma(\mathbb Z_n),k)=\binom nk\qquad(m\le k\le n),
		\end{equation}
		and
		\begin{equation}\label{eq next coefficient}
			d(\Gamma(\mathbb Z_n),m-1)=\binom n{m-1}-c.
		\end{equation}
		In particular, \emph{every} $k$-element vertex subset is dominating if and only if $k\ge m$, where $0\le k\le n$.
	\end{theorem}
	
	\begin{proof}
		Let $U=U(\mathbb Z_n)$ and
		$H_n=\Gamma(\mathbb Z_n)[V\setminus U]$.
		Every unit is adjacent to every other vertex,
		so any nonempty set meeting $U$ is dominating.
		Since $|V(H_n)|=m$, every vertex subset of
		cardinality $k>m$ meets $U$. For $k=m$, the
		only subset avoiding $U$ is $V(H_n)$, which
		also dominates $\Gamma(\mathbb Z_n)$, as
		every unit is adjacent to every nonunit.
		Consequently,
		$
		d(\Gamma(\mathbb Z_n),k)=\binom nk,$ and $m\leq k\leq n.
		$
		
		Now consider a subset $S$ of cardinality
		$m-1$. If $S$ contains a unit, it is dominating, otherwise, $S=V(H_n)\setminus\{v\}$ for
		some $v\in V(H_n)$. All vertices outside
		$S$, except possibly $v$, are dominated
		by $S$. Hence $S$ fails to dominate precisely
		when $v$ is isolated in $H_n$. By
		\eqref{eq master isolate}, $H_n$ has exactly
		$c$ isolated vertices, giving
		$$
		d(\Gamma(\mathbb Z_n),m-1)
		=\binom n{m-1}-c.
		$$
		
		Finally, $c\geq1$ ensures the existence of
		a non-dominating subset of size $m-1$.
		Every subset of such a set is likewise
		non-dominating, since domination is
		preserved under taking supersets.
		Thus, for each $0\leq k<m$, some
		$k$-element subset fails to dominate. 
	\end{proof}
	
	The following consequence discusses the log-concavity of the upper tail $D(\Gamma(\mathbb Z_n),x)$.
	\begin{corollary}\label{cor tail}
		The coefficient sequence of $D(\Gamma(\mathbb Z_n),x)$ is strictly log-concave at each index $m\le k\le n-1$ for which both neighbouring indices are defined:
		$$d(\Gamma(\mathbb Z_n),k)^2>
		d(\Gamma(\mathbb Z_n),k-1)d(\Gamma(\mathbb Z_n),k+1).$$
		Furthermore, the sequence of coefficients from degree $m$ to degree $n$ follows the corresponding binomial coefficients and is therefore unimodal.
	\end{corollary}

	\begin{proof}
		By \eqref{eq top coeff}, we have
		$
		d(\Gamma(\mathbb Z_n),k)=\binom nk,~m\leq k\leq n.
		$
		For $m+1\leq k\leq n-1$, strict
		log-concavity follows from
		$$
		\frac{\binom nk^2}
		{\binom n{k-1}\binom n{k+1}}
		=\frac{(k+1)(n-k+1)}{k(n-k)}>1.
		$$
		At $k=m$, the neighbouring coefficients
		satisfy
		$$
		d(\Gamma(\mathbb Z_n),m-1)
		=\binom n{m-1}-c<\binom n{m-1}
		$$
		by \eqref{eq next coefficient}, while
		$d(\Gamma(\mathbb Z_n),m)=\binom nm$ and
		$d(\Gamma(\mathbb Z_n),m+1)=\binom n{m+1}$.
		Consequently,
		$$
		\begin{aligned}
			d(\Gamma(\mathbb Z_n),m)^2
			&=\binom nm^2 >\binom n{m-1}\binom n{m+1} \geq d(\Gamma(\mathbb Z_n),m-1)
			d(\Gamma(\mathbb Z_n),m+1).
		\end{aligned}
		$$
		Thus, strict log-concavity holds throughout
		the asserted range. Finally, the coefficients
		from degree $m$ to degree $n$ form a consecutive
		segment of the binomial sequence
		$\left \{\binom nk\right \}_{k=0}^{n}$, which is unimodal.
		Hence the upper coefficient tail is also unimodal.
	\end{proof}
	
	\medskip
	The following exact evaluations measure the number of subsets that do not dominate. They also give simple checks on symbolic expansions.
	\begin{theorem}\label{thm nondom}
		Let $N_n$ denote the number of non-dominating subsets of $V(\Gamma(\mathbb Z_n))$, including the empty set. Then
		\begin{equation}\label{eq count nondom}
			N_n=2^{n-\phi(n)}-D(H_n,1).
		\end{equation}
		In particular,
		\begin{itemize}
			\item[\bf (i)] If $n=p^a$, then $N_n=2^{p^{a-1}}-1$.
			\item[\bf (ii)] If $n=p^a q^b$, with the parameters $c,A,B$ of Theorem \ref{domination poly of zn n=pq}, then
			\begin{equation}\label{eq N two}
				N_n=(2^c-1)2^{A+B}+2^A+2^B-3.
			\end{equation}
		\end{itemize}
		Consequently the probability that a uniformly chosen subset of vertices is dominating equals $1-N_n/2^n$.
	\end{theorem}

	\begin{proof}
		Let $U=U(\mathbb Z_n)$ and
		$H_n=\Gamma(\mathbb Z_n)[V\setminus U]$.
		Since every unit is adjacent to all other vertices,
		each subset meeting $U$ is dominating. Thus,
		a non-dominating set must lie entirely in $H_n$.
		Moreover, a subset of $V(H_n)$ dominates
		$\Gamma(\mathbb Z_n)$ precisely when it dominates
		$H_n$, since every nonempty subset of $V(H_n)$
		also dominates all units. As $|V(H_n)|=n-\phi(n)$,
		the number of non-dominating subsets is
		$
		N_n=2^{n-\phi(n)}-D(H_n,1),
		$
		which proves \eqref{eq count nondom}.
		
		If $n=p^a$, then
		$H_n\cong\overline K_{p^{a-1}}$. Its only
		dominating set is the entire vertex set, so
		$D(H_n,1)=1$. Consequently,
		$
		N_n=2^{p^{a-1}}-1.
		$
		
		Now suppose $n=p^aq^b$. By the structural
		decomposition in Theorem~\ref{domination poly of zn n=pq},
		$
		H_n\cong\overline K_c\cup K_{A,B},$ and $n-\phi(n)=c+A+B.
		$
		A subset dominates $K_{A,B}$ precisely when
		it meets both parts or contains one part
		entirely. Hence
		$
		D(H_n,1)=(2^A-1)(2^B-1)+2.
		$
		Substituting these observations \eqref{eq count nondom}, we have
		\begin{align*}
			N_n
			&=2^{c+A+B}-(2^A-1)(2^B-1)-2 =(2^c-1)2^{A+B}+2^A+2^B-3,
		\end{align*}
		establishing \eqref{eq N two}.
		
		Finally, the vertex set of $\Gamma(\mathbb Z_n)$
		has $2^n$ subsets, all equally likely under
		uniform selection. Since exactly $N_n$ are
		non-dominating, the required probability is
		$1-\tfrac{N_n}{2^n}$.
	\end{proof}

	The following result determines the value of the domination polynomial at $x=-1$ for prime powers and integers having two distinct prime divisors. In particular, it shows that $-1$ is not a domination root for these families.
	
	\begin{proposition}\label{prop minusone}
		If $n=p^a$ is a prime power, then
		$
		D(\Gamma(\mathbb Z_n),-1)=(-1)^{p^{a-1}}.
		$
		If $n=p^aq^b$, with $p\ne q$ primes, then
		\begin{equation}\label{eq minusone}
			D(\Gamma(\mathbb Z_n),-1)
			=(-1)^c\big(1+(-1)^A+(-1)^B\big).
		\end{equation}
		Thus $-1$ is not a domination root when $n$ has at
		most two distinct prime divisors.
	\end{proposition}
	
	\begin{proof}
		By \eqref{eq master}, since $\phi(n)>0$ and
		$n-\phi(n)>0$, evaluating at $x=-1$ gives
		$
		D(\Gamma(\mathbb Z_n),-1)=D(H_n,-1).
		$
		For $n=p^a$, the nonunit-induced graph is
		$\overline{K}_{p^{a-1}}$, and hence
		$
		D(H_n,-1)=(-1)^{p^{a-1}}.
		$
		For $n=p^aq^b$, equation \eqref{eq pq full}
		gives
		$$
		D(H_n,x)=x^c\left(
		\big((1+x)^A-1\big)\big((1+x)^B-1\big)
		+x^A+x^B\right).
		$$
		Since $A,B\geq1$, substitution of $x=-1$ implies
		$$
		D(H_n,-1)=(-1)^c
		\big(1+(-1)^A+(-1)^B\big).
		$$
		The expression in parentheses belongs to
		$\{-1,1,3\}$ and is therefore nonzero.
		The prime-power value is also nonzero,
		proving the assertion.
	\end{proof}

	\begin{table}[H]
		\centering\small
		\begin{tabular}{@{}rrrrrr@{}}\toprule
			$n$&$\phi(n)$&$c$&$m=n-\phi(n)$&$[x]D$&$N_n$\\\midrule
			$4$&$2$&$2$&$2$&$2$&$3$\\
			$6$&$2$&$1$&$4$&$2$&$11$\\
			$8$&$4$&$4$&$4$&$4$&$15$\\
			$12$&$4$&$2$&$8$&$4$&$209$\\
			$15$&$8$&$1$&$7$&$8$&$81$\\
			$18$&$6$&$3$&$12$&$6$&$3653$\\
			$30$&$8$&$1$&$22$&$8$&$2667195$\\
			$32$&$16$&$16$&$16$&$16$&$65535$\\
			$36$&$12$&$6$&$24$&$12$&$16519229$\\
			\bottomrule
		\end{tabular}
		\caption{Numerical values for selected $n$. The column $N_n$ counts non-dominating subsets, including the empty set, and all values follow from Theorem \ref{thm nondom}.}
		\label{tab basic}
	\end{table}

	Table~\ref{tab basic} presents the basic arithmetic parameters, the number of dominating vertices, and the total number of non-dominating subsets of $\Gamma(\mathbb Z_n)$ for selected values of $n$, illustrating the results of Theorem~\ref{thm nondom}.

	\section{Additional structural and enumerative results}\label{section additional}
	
	The prime-support description also relates moduli having the same set of prime divisors.  The following result makes that relation precise, since it is useful when the exponents in the prime factorization vary while the underlying square-free modulus stays fixed.
	
	\begin{theorem}\label{thm radical lifting}
		Let $n=\prod_{i=1}^r p_i^{a_i}$, put $s=\operatorname{rad}(n)$ and $c=n/s$, and let $H_s$ and $H_n$ be the nonunit-induced graphs defined above.  Then $H_n$ is isomorphic to the \emph{independent $c$-fold blow-up} of $H_s$: every vertex of $H_s$ is replaced by an independent set of $c$ vertices, with complete adjacency between two replacement sets precisely when the original vertices are adjacent.  In particular, the isomorphism type of $H_n$ depends on the exponents $a_i$ only through $c$ once $s$ is fixed.
		
		More generally, for any finite simple graph $F$ and integer $c\ge1$, let $F^{(c)}$ be its independent $c$-fold blow-up.  For a subset $T\subseteq V(F)$, write $N_F(T)=\bigcup_{v\in T}N_F(v)$ for its open neighbourhood, and put $N_F(v)=\{w:vw\in E(F)\}$.  Then
		\begin{equation}\label{eq radical polynomial}
			D(F^{(c)},x)=
			\sum_{\substack{T\subseteq V(F)\\ V(F)\setminus T\subseteq N_F(T)}}
			\ \prod_{v\in T}
			\begin{cases}
				(1+x)^c-1,&N_F(v)\cap T\ne\varnothing,\\
				x^c,&N_F(v)\cap T=\varnothing.
			\end{cases}
		\end{equation}
		Consequently, $D(\Gamma(\mathbb Z_n),x)$ may be obtained by applying \eqref{eq radical polynomial} with $F=H_s$ and then using \eqref{eq master}.
	\end{theorem}
	
	\begin{proof}
		For each nonempty $S\subseteq[r]$, let $B_S(n)$
		denote the set of nonunit residues having prime
		support $S$. By \eqref{eq weights}, we have
		$$
		|B_S(n)|=c\prod_{j\notin S}(p_j-1)
		=c|B_S(s)|.
		$$
		As each block is independent, and distinct blocks
		$B_S(n)$ and $B_T(n)$ are completely joined
		precisely when $S\cap T=\varnothing$. The same
		adjacency rule holds for $H_s$. Hence each
		$B_S(n)$ can be partitioned into $|B_S(s)|$
		independent sets of size $c$, with the
		adjacencies between these sets agreeing
		exactly with those of the corresponding
		vertices of $H_s$. This proves
		$H_n\cong H_s^{(c)}$, including the isolated
		vertices corresponding to $S=[r]$.
		
		Now let $F$ be any finite simple graph and
		$X$ a dominating set of $F^{(c)}$. For
		$v\in V(F)$, denote its replacement set by
		$C_v$ and let
		$$
		T=\{v\in V(F):X\cap C_v\neq\varnothing\}.
		$$
		If $v\notin T$, all vertices of $C_v$ must
		be dominated from a selected neighbouring
		replacement set. Thus $T$ necessarily satisfies
		$
		V(F)\setminus T\subseteq N_F(T).
		$
		Conversely, fix a subset $T$ satisfying this
		condition. For $v\in T$, if
		$N_F(v)\cap T\neq\varnothing$, the selected
		vertices in a neighbouring replacement set
		dominate every unselected vertex of $C_v$.
		Consequently, any nonempty subset of $C_v$
		may be chosen, contributing $(1+x)^c-1$.
		If $N_F(v)\cap T=\varnothing$, no selected
		vertex outside $C_v$ is adjacent to its
		vertices. Since $C_v$ is independent, the
		entire set $C_v$ must be selected, contributing
		$x^c$.
		
		These conditions are sufficient for domination
		and the choices are independent across
		replacement sets. Moreover, every dominating
		set determines a unique $T$, so summing the
		corresponding products gives
		\eqref{eq radical polynomial} without
		repetition. Finally, using
		$$
		\Gamma(\mathbb Z_n)\cong
		K_{\phi(n)}\vee H_n
		\cong K_{\phi(n)}\vee H_s^{(c)}
		$$
		in \eqref{eq master} the required 
		formula for $D(\Gamma(\mathbb Z_n),x)$ follows.
	\end{proof}

	\medskip
	The next theorem refines the domination-number formula in the two-prime case by giving the number of minimum nonunit dominating sets.  Although the full co-maximal graph has domination number one, this coefficient measures the first possible contribution from dominating sets that contain no units.
	
	\begin{theorem}\label{thm twoprime minimum}
		Let $n=p^a q^b$, where $p<q$ are primes, $a,b\ge1$, and let $c=n/(pq)$, $A=c(q-1)$, $B=c(p-1)$ as in Theorem \ref{domination poly of zn n=pq}.  Put $g=\gamma(H_n)$.  Then
		\begin{equation}\label{eq twoprime minimum}
			(g,d(H_n,g))=
			\begin{cases}
				(c+1,1),&B=1,\\
				(c+2,AB+\mathbf1_{\{A=2\}}+\mathbf1_{\{B=2\}}),&B\ge2,
			\end{cases}
		\end{equation}
		where $\mathbf1_{\{P\}}$ equals $1$ if the property $P$ holds and $0$ otherwise.  The first case $B=1$ occurs precisely when $n=2q$ for an odd prime $q$.  In particular, the first term of the difference
		$$D(\Gamma(\mathbb Z_n),x)-\big((1+x)^n-(1+x)^{n-\phi(n)}\big)$$
		is exactly $d(H_n,g)x^g$.
	\end{theorem}
	
	\begin{proof}
		By the two-prime decomposition,
		$
		H_n\cong\overline K_c\cup K_{A,B},$ with $A=c(q-1)\ge B=c(p-1).
		$
		Every dominating set of $H_n$ contains all
		$c$ isolated vertices. Consequently,
		$
		\gamma(H_n)=c+\gamma(K_{A,B}),
		$
		and the minimum dominating sets of $H_n$
		correspond bijectively to those of $K_{A,B}$. If $B=1$, then $A\ge2$, and the unique vertex
		in the part of size one dominates $K_{A,B}$.
		No vertex in the other part is dominating,
		so $g=c+1$ and $d(H_n,g)=1$. Now suppose $B\ge2$. Neither part contains
		a dominating vertex, whereas any pair with
		one vertex from each part dominates $K_{A,B}$.
		Hence $g=c+2$, and there are $AB$ such pairs.
		A pair contained entirely in one part is
		dominating precisely when that part has size
		two; otherwise, an unselected vertex of the
		same part remains undominated. These additional
		pairs contribute
		$\mathbf1_{\{A=2\}}+\mathbf1_{\{B=2\}}$, where   $\mathbf{1}_{\{P\}}$ is the indicator function of a property $P$.		
		Thus,
		$
		d(H_n,g)=AB+\mathbf1_{\{A=2\}}
		+\mathbf1_{\{B=2\}},
		$ so \eqref{eq twoprime minimum} follows. 
		
		Furthermore, $B=c(p-1)=1$ holds exactly when
		$c=1$ and $p=2$, equivalently $n=2q$.
		Finally, \eqref{eq master} gives
		$$
		D(\Gamma(\mathbb Z_n),x)
		-\big((1+x)^n-(1+x)^{n-\phi(n)}\big)
		=D(H_n,x).
		$$
		Since the least nonzero term of $D(H_n,x)$
		has degree $g$ and coefficient $d(H_n,g)$,
		the final assertion follows.
	\end{proof}
	
	\medskip
	The sharp threshold in Theorem \ref{thm top} gives the two highest coefficient regimes.  One more coefficient can be determined, with an exceptional correction arising exactly when the modulus is twice an odd prime.
	\begin{theorem}\label{thm top second}
		Let $n\ge4$ be composite, and put $m=n-\phi(n)$ and $c=n/\operatorname{rad}(n)$.  Then
		\begin{equation}\label{eq mminus2}
			\begin{split}
				d(\Gamma(\mathbb Z_n),m-2)={}&\binom n{m-2}-\binom m2+\binom{m-c}{2} -\begin{cases}q-1,&n=2q\text{ with }q\text{ an odd prime},\\0,&\text{otherwise}.\end{cases}
			\end{split}
		\end{equation}
		Thus the coefficient at degree $m-2$ detects the exceptional degree-one vertices of the nonunit subgraph.
	\end{theorem}
	
	\begin{proof}
		Let $U=U(\mathbb Z_n)$ and $H_n$ be the
		nonunit-induced graph, so that $|V(H_n)|=m$.
		Every subset meeting $U$ is dominating.
		Hence a non-dominating set of cardinality
		$m-2$ must have the form
		$
		S=V(H_n)\setminus\{v,w\},~v\neq w.
		$
		Such a set fails to dominate precisely when
		$
		N_{H_n}(v)\subseteq\{w\}
		$ or $
		N_{H_n}(w)\subseteq\{v\}.
		$
		Thus, an omitted pair is inadmissible if it
		contains an isolated vertex or contains a
		degree-one vertex together with its unique
		neighbour.
		
		By \eqref{eq master isolate}, $H_n$ has exactly
		$c$ isolated vertices. The number of omitted
		pairs containing at least one of them is
		$
		\binom m2-\binom{m-c}{2}.
		$
		It remains to count inadmissible pairs among
		the nonisolated vertices. 	If $n$ is a prime power, all vertices of $H_n$
		are isolated. For $n=p^aq^b$, with $p<q$,
		the nonisolated subgraph is $K_{A,B}$, where
		$
		A=c(q-1),$ and $B=c(p-1).
		$
		Its minimum degree is $B$, which equals one
		precisely when $c=1$ and $p=2$, or equivalently
		$n=2q$. In this exceptional case, the
		nonisolated subgraph is $K_{q-1,1}$, and
		its $q-1$ edges give exactly $q-1$
		additional inadmissible omitted pairs.
		
		Now suppose $r\geq3$. Every nonisolated
		vertex in a support block $B_S$, with
		$\varnothing\neq S\subsetneq[r]$, is adjacent
		to every vertex of $B_{\{i\}}$ whenever
		$i\notin S$. Moreover,
		$$
		|B_{\{i\}}|
		=c\prod_{j\neq i}(p_j-1)\geq2,
		$$
		since at least two distinct primes occur in
		the product and at most one of them equals
		$2$. Thus every nonisolated vertex has
		degree at least two, and no further
		inadmissible pairs arise. The two classes of inadmissible pairs are
		disjoint. Subtracting their total number
		from the $\binom n{m-2}$ vertex subsets
		of cardinality $m-2$ gives
		$$
		\begin{aligned}
			d(\Gamma(\mathbb Z_n),m-2)
			={}&\binom n{m-2}-\binom m2
			+\binom{m-c}{2}-\begin{cases}
				q-1,&n=2q,\quad q\text{ an odd prime},\\
				0,&\text{otherwise}.
			\end{cases}
		\end{aligned}
		$$ 
	\end{proof}
	
	\medskip
	For the two-prime family, it is also possible to count the non-dominating subsets \emph{at every size}, rather than only their total number as in Theorem \ref{thm nondom}.
	
	\begin{theorem}\label{thm nondom polynomial}
		Let $n=p^a q^b$ with $c,A,B$ as above, and let $F_n(x)=\sum_{k=0}^n f_kx^k$, where $f_k$ counts the vertex subsets of cardinality $k$ that do not dominate $\Gamma(\mathbb Z_n)$.  Then
		\begin{equation}\label{eq nondom polynomial}
			\begin{split}
				F_n(x)={}&\big((1+x)^c-x^c\big)(1+x)^{A+B}+x^c\big[(1+x)^A+(1+x)^B-1-x^A-x^B\big].
			\end{split}
		\end{equation}
		Consequently, for every $0\le k\le n$,
		\begin{equation}\label{eq nondom coefficient}
			\begin{split}
				f_k={}&\binom{c+A+B}{k}-\binom{A+B}{k-c}+\binom{A}{k-c}+\binom{B}{k-c}-\mathbf1_{\{k=c\}}-\mathbf1_{\{k=c+A\}}-\mathbf1_{\{k=c+B\}},
			\end{split}
		\end{equation}
		where binomial coefficients outside their natural ranges are zero.  In particular, $F_n(1)$ recovers \eqref{eq N two}.
	\end{theorem}
	
	\begin{proof}
		Since every unit is a dominating vertex, each
		non-dominating subset lies entirely in $H_n$.
		By the two-prime decomposition,
		$
		H_n\cong\overline K_c\cup K_{A,B}.
		$
		We partition the non-dominating subsets according
		to whether they contain all $c$ isolated vertices.
		
		If at least one isolated vertex is omitted, the
		subset cannot dominate $H_n$. The generating
		polynomial for these choices is
		$
		\big((1+x)^c-x^c\big)(1+x)^{A+B}.
		$
		If all isolated vertices are selected, domination
		depends entirely on the selected vertices of
		$K_{A,B}$. A subset of this bipartite graph
		fails to dominate precisely when it is empty
		or is a proper nonempty subset of exactly one
		part. These possibilities contribute
		$
		(1+x)^A+(1+x)^B-1-x^A-x^B.
		$
		Including the isolated vertices, we have
		$
		x^c\big[(1+x)^A+(1+x)^B-1-x^A-x^B\big].
		$
		The two cases are disjoint and exhaust all
		non-dominating subsets, establishing
		\eqref{eq nondom polynomial}.
		
		Expanding the first term as
		$
		(1+x)^{c+A+B}-x^c(1+x)^{A+B}
		$
		and extracting the coefficient of $x^k$, we have
		\begin{align*}
			f_k={}&\binom{c+A+B}{k}
			-\binom{A+B}{k-c}
			+\binom{A}{k-c}+\binom{B}{k-c}-\mathbf1_{\{k=c\}}
			-\mathbf1_{\{k=c+A\}}
			-\mathbf1_{\{k=c+B\}},
		\end{align*}
		which proves \eqref{eq nondom coefficient}.
		Finally, for $x=1$, we have
		$
		F_n(1)=(2^c-1)2^{A+B}+2^A+2^B-3,
		$ which agrees with \eqref{eq N two}.
	\end{proof}

	\section{Unimodal and log-concave properties of $\Gamma(\mathbb Z_n)$}\label{section 3}
	A real sequence $e_0,e_1,\ldots,e_d$ is \emph{unimodal} if there exists an index $\ell$ such that
	$$e_0\le e_1\le\cdots\le e_\ell\ge e_{\ell+1}\ge\cdots\ge e_d.$$
	The index $\ell$ is a \emph{mode}, and  it need not be unique. A polynomial is called unimodal when its coefficient sequence is unimodal. In this discussion an \emph{oscillation} means a change between strict increase and strict decrease after constant stretches have been ignored, so a unimodal sequence has at most one such change.
	
	The sequence is \emph{log-concave} if $e_k^2\ge e_{k-1}e_{k+1}$ for $1\le k\le d-1.$ A nonnegative log-concave sequence with no internal zeros is unimodal, but unimodality alone does not imply log-concavity. If a polynomial with nonnegative coefficients has only real zeros, Newton's inequalities give the stronger estimates
	\begin{equation*}\label{log con 2}
		e_k^2\ge e_{k-1}e_{k+1}\frac{k+1}{k}\frac{d-k+1}{d-k},\qquad1\le k\le d-1;
	\end{equation*}
	see \cite{hardy,stanley}. The products of polynomials with nonnegative log-concave coefficient sequences and no internal zeros again have log-concave coefficient sequences \cite{stanley}. These principles are prominent in work on independence and matching polynomials \cite{levitsurvey,schwenk} and in the theory of chromatic polynomials, and the log-concavity of the absolute chromatic coefficients was proved by Huh \cite{huh}. The general coefficient problem for domination polynomials is discussed in \cite{brown1,beaton}.
	 For example, $1+3x+4x^2+5x^3+2x^4+x^5$ is unimodal but is not log-concave, since at index $4$ its coefficient is $2$ and $2^2<5\cdot1$. 
	
	For prime $p$, Theorem \ref{domination poly of zn} gives $D(\Gamma(\mathbb Z_p),x)=(1+x)^p-1$. Its nonzero coefficient sequence is the binomial row $\binom p1,\ldots,\binom pp$, so the following statement is immediate.
	\begin{proposition}\label{prop prime lc}
		If $n$ is prime, the domination polynomial $D(\Gamma(\mathbb Z_n),x)$ is log-concave and unimodal.
	\end{proposition}
	\begin{proof}
		For $1\le k\le n$, the coefficients are $\binom nk$. Their log-concavity follows from the binomial ratio identities. The initial zero coefficient does not affect the inequality at $k=1$, and the nonzero coefficients have no internal zeros, giving unimodality.
	\end{proof}
	
	\medskip
	For composite $n$, a conclusion for the full coefficient sequence requires additional work. In particular, although $((1+x)^{a(p-1)}-1)(1+x)^a$ is a product of two log-concave polynomials, the additional term $x^a$ in \eqref{eq pp} cannot be incorporated by a general result that sums of log-concave sequences are log-concave. Similarly, the two-prime formula \eqref{eq pq full} contains a binomial difference and a separate polynomial contribution. Theorem \ref{thm top} and Corollary \ref{cor tail} establish the upper-tail properties without relying on these invalid closure steps.
	
	Exact symbolic expansion confirms log-concavity at $n=4,6,8,9,12,15,18$ and at $n=25,27,30,32,36$. This finite calculation is evidence for the following question, not a proof of a statement for all prime powers or all integers with two prime divisors.
	\begin{question}\label{question lc}
		For which integers $n$ is the full coefficient sequence of $D(\Gamma(\mathbb Z_n),x)$ log-concave? Can unimodality be established uniformly from the prime-support decomposition without further hypotheses on the prime factors?
	\end{question}
	
	\medskip
	For a polynomial $P(x)=\sum_{j=0}^d a_jx^j$ with all $a_j>0$, the Enestr\"om--Kakeya theorem implies that every complex zero $z$ satisfies
	\begin{equation}\label{eq EK}
		\min_{0\le j<d}\frac{a_j}{a_{j+1}}\le|z|\le
		\max_{0\le j<d}\frac{a_j}{a_{j+1}}.
	\end{equation}
	See \cite{barbeau,rahman}. Since a domination polynomial always vanishes at zero, this coefficient-ratio bound must be applied to the polynomial after dividing out its factor $x^{\gamma(G)}$. Here $\gamma(\Gamma(\mathbb Z_n))=1$, so division by $x$ is enough.
	
	Next, we have the uniform annulus for nonzero roots for $D(\Gamma(\mathbb Z_n),x).$
	\begin{theorem}\label{thm zeros}
		For every $n\ge2$, let $d_k=d(\Gamma(\mathbb Z_n),k)$. All coefficients $d_k$ with $1\le k\le n$ are positive. Every nonzero root $z$ of $D(\Gamma(\mathbb Z_n),x)$ satisfies
		\begin{equation}\label{eq zeros revised}
			0<r_n\le |z|\le R_n=n,
			\qquad r_n=\min_{1\le k\le n-1}\frac{d_k}{d_{k+1}}.
		\end{equation}
		The value $R_n=n$ is the exact maximum of these consecutive-coefficient ratios.
	\end{theorem}
	\begin{proof}
		For any unit $u\in\mathbb Z_n$, every vertex set containing $u$ is dominating, so $d_k>0$ for $1\le k\le n$. Therefore the polynomial $D(\Gamma(\mathbb Z_n),x)/x$ has positive coefficients and \eqref{eq EK} makes sense. To identify the upper ratio, double count pairs $(S,T)$ where $S$ is a dominating $k$-set and $T$ is a $(k+1)$-set containing $S$. There are $(n-k)d_k$ such pairs, and a $(k+1)$-set contains at most $k+1$ distinct $k$-subsets. Since a superset of a dominating set also dominates,
		$$(n-k)d_k\le(k+1)d_{k+1},\qquad
		\frac{d_k}{d_{k+1}}\le\frac{k+1}{n-k}\le n.$$
		By Theorem \ref{thm top}, $d_{n-1}=n$ and $d_n=1$, so the ratio at $k=n-1$ is $n$. The lower ratio is positive and the result follows.
	\end{proof}
	
	\begin{remark}
		The radius $n$ in \eqref{eq zeros revised} is a rigorous Enestr\"om--Kakeya bound but need not be close to the largest actual root modulus. Any proposed smaller constant obtained from a particular ratio requires an independent proof that it bounds \emph{every} ratio or every root. A unimodal coefficient sequence does not, by itself, locate the largest consecutive-coefficient ratio.
	\end{remark}
	
	\begin{example}\label{ex root32}
		For $D(\Gamma(\mathbb Z_{32}),x)=(1+x)^{32}-(1+x)^{16}+x^{16}$, exact coefficient arithmetic gives
		$r_{32}=\frac{16}{376}=\frac{2}{47},$ and $R_{32}=32.$
		Thus every nonzero root lies in $2/47\le|z|\le32$. The numerical root computation gives one negative real root approximately $-0.4999997616$, the simple root $0$, and fifteen pairs of nonreal conjugate roots. The largest computed modulus is approximately $2.608996$, see Figure \ref{Fig zeros}(left).
	\end{example}
	
	\begin{example}\label{ex root15}
		The coefficients in Example \ref{ex 15} give
		$r_{15}=\tfrac{8}{84}=\tfrac{2}{21},$ and $R_{15}=15.$
		The polynomial has one root at $0$ and seven pairs of nonreal conjugate roots. The largest numerical modulus is approximately $2.312975$, see Figure \ref{Fig zeros}(right).
	\end{example}
	
	\begin{figure}[H]
		\centering
		{\scalebox{.3}{\includegraphics{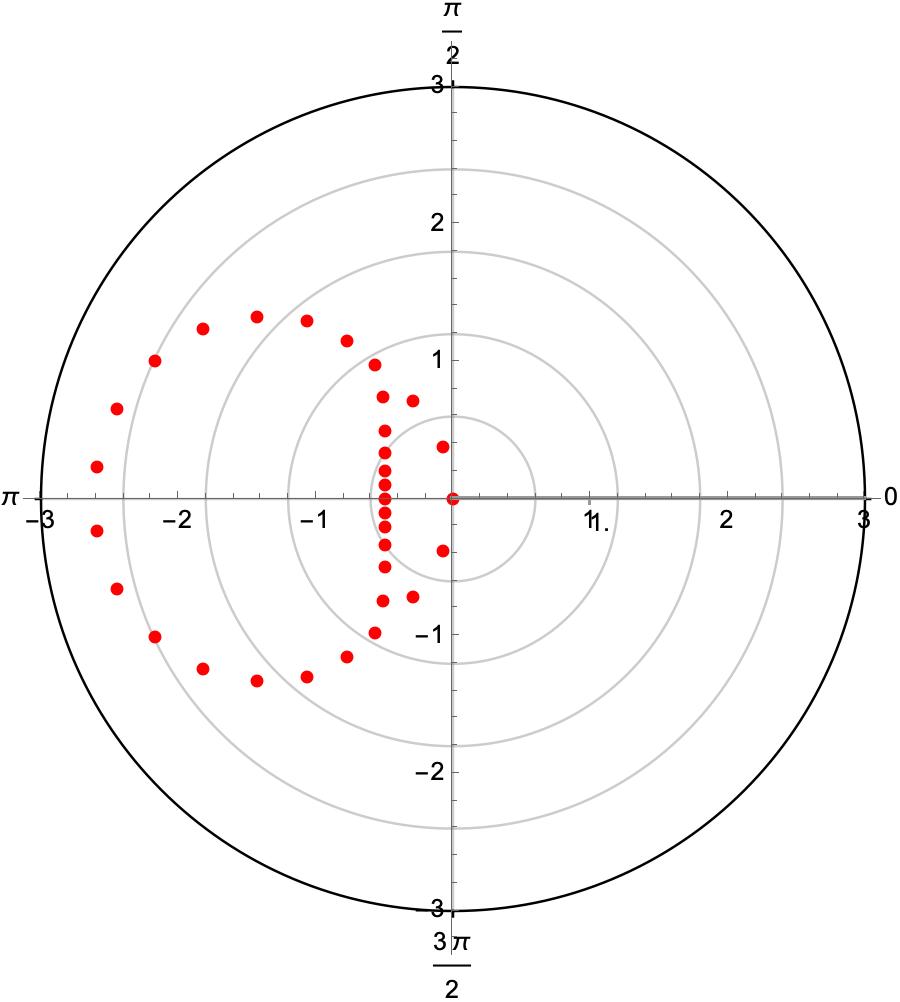}}\qquad\scalebox{.25}{\includegraphics{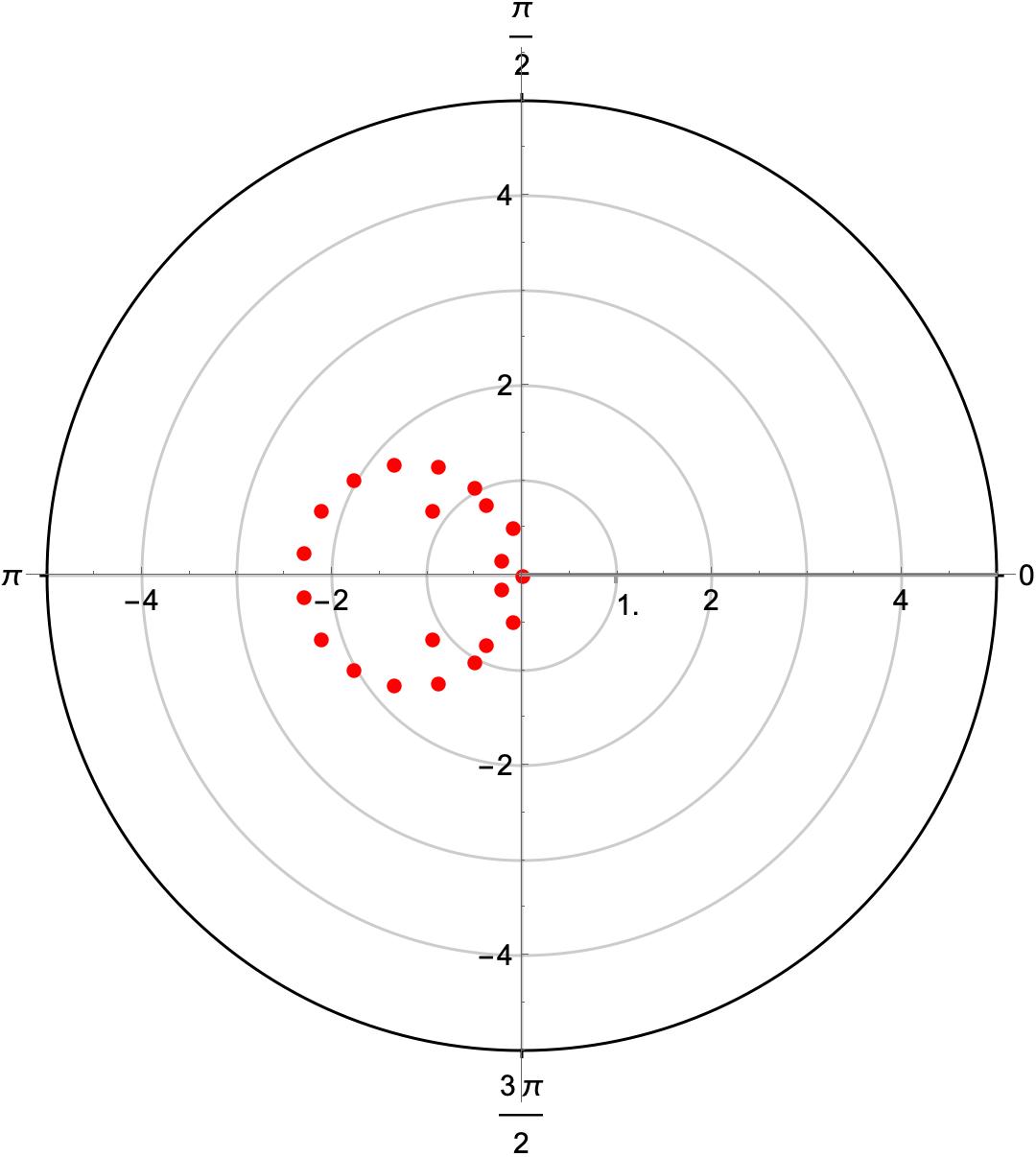}}}
		\caption{Representation of zeros of $D(\Gamma(\mathbb{Z}_{2^{5}}),x)$ and $D(\Gamma(\mathbb{Z}_{3\cdot 7}),x)$ on plane.}
		\label{Fig zeros}
	\end{figure}
	
	\section*{Conclusion}
	The unit clique in $\Gamma(\mathbb Z_n)$ allows domination-polynomial enumeration to be reduced to the induced nonunit graph.  The prime-support expression treats every positive prime-factorization pattern and gives exact nonunit domination numbers and, for at least three distinct primes, counts of minimum nonunit dominating sets. The highest coefficients and the number of non-dominating vertex subsets are determined, and nonzero domination roots satisfy a justified annular estimate. A radical-lifting identity describes the dependence on prime-power exponents, the first nonunit coefficient for two prime divisors, the coefficient immediately preceding the known binomial tail, and the non-domination generating polynomial are also calculated. The coefficients of the upper tail are rigorously log-concave and unimodal. Establishing full log-concavity in broader classes and obtaining sharper root bounds are natural further problems.
	
	\section*{Declarations}
	\noindent\textbf{Acknowledgement:} The numerical computations presented in this article were performed using Mathematica and Python 3. The figures were prepared in LaTeX using the TikZ package, with AI-assisted tools employed in generating the corresponding code. In particular, Figure \ref{Fig zeros}, illustrating the distribution of domination roots in the complex plane, was generated using Mathematica.\\
	\noindent\textbf{Data Availability:} No external dataset was used. The numerical values and root plots can be reproduced from the exact polynomial formulas in the article. \\
	\noindent\textbf{Funding:} The author received no financial support for this work.\\
	\noindent\textbf{Conflict of interest:} The author declares no competing interests.

\end{document}